\documentclass[12pt]{article}
\usepackage{amsmath,amsfonts,amsthm,amssymb, graphicx,longtable, color, algorithm, algorithmic}
\usepackage[shortlabels]{enumitem}
\usepackage[hmargin={1.0in,1.0in},vmargin={1.0in,1.0in}]{geometry}
\usepackage{setspace}
\usepackage{float}
\usepackage{subfigure}
\usepackage{array}
\usepackage{soul}
\usepackage{multirow}

\providecommand{\U}[1]{\protect\rule{.1in}{.1in}}
\providecommand{\U}[1]{\protect\rule{.1in}{.1in}}
\theoremstyle{plain}
\newtheorem{thm}{Theorem}
\newtheorem{proposition}{Proposition}

\newtheorem{corollary}{Corollary}
\newtheorem{assumption}{Assumption}
\newtheorem{remark}{Remark}
\newtheorem{definition}{Definition}

\renewenvironment{proof}[1][Proof]{\noindent\textbf{#1.} }{\ \rule{0.5em}{0.5em}}

\begin{document}
\title{Perfect Sampling of Hawkes Processes and Queues with Hawkes Arrivals}
\date{}
\author{Xinyun Chen\\
	Institute for Data and Decision Analytics\\
	Chinese University of Hong Kong, Shenzhen\\
	 Email: \texttt{chenxinyun@cuhk.edu.cn}}
\maketitle
\begin{abstract}
	In this paper we develop the first perfect sampling algorithm for queues with Hawkes input, i.e. single-server queues with Hawkes arrivals and i.i.d. service times of general distribution. In addition to the stability condition, we also assume the excitation function of the Hawkes process has a light tail and the service time has finite moment generating function in the neighborhood of the origin. In this procedure, we also propose a new perfect sampling algorithm for Hawkes process with improved computational efficiency compared to the existing algorithm. Theoretical analysis and numerical tests on the algorithms' correctness and efficiency are also included.
\end{abstract}
\section{Introduction}\label{sec: intro}
Many stochastic systems have arrival processes that exhibit clustering or self-exciting behavior, i.e.  an arrival will increase the possibility of new arrivals.  As a natural extension of the classic Poisson process, Hawkes processes are used widely to model arrivals with self-excitement. Examples include order flows in stock market (\cite{Abergel_SIAM_2015}, \cite{trade}), risk events in financial systems (\cite{finance}, \cite{portfolio}) and social network events (\cite{nips}, \cite{tweet}). 

To study the impact of self-excitement on the performance of stochastic systems, several papers have analyzed queueing dynamics with customer arrivals following a Hawkes, or other type of self-exiting process. \cite{Gao_QS_2018} studies the heavy-traffic limit of infinite-server queue with Hawkes arrivals. \cite{Pender_SS_2018} and \cite{Koops_JAP_2018} provide analytic solution to the transient and steady-state moments for different infinite-server systems with Hawkes arrivals. In particular, \cite{Pender_SS_2018} studies the systems with Markovian Hawkes arrivals and phase-type/deterministic service times, while \cite{Koops_JAP_2018} studies the cases with non-Markovian Hawkes arrivals and exponential service times. In addition, \cite{Koops_QS_2017} studies an infinite-server queue with shot-noise arrivals. \cite{Pender_wp_2018} proposes a so-called Queue-Hawkes model that combining Hawkes process with an infinite-server queue to capture ephemeral self-exciting behaviors. To the best of our knowledge, analytic results on queueing processes with Hawkes arrivals are only available for infinite-server systems in the literature. Due to the dependence between customer arrivals and sojourn times, it is difficult to obtain analytic results even for the most simple single-server queue with Hawkes arrivals (see, for instance, the discussion on page 941 of \cite{Koops_JAP_2018}). 

In this paper, we apply simulation techniques to numerically compute the steady state of queues with Hawkes input. In detail, we develop the a perfect sampling algorithm that can generate i.i.d. samples exactly from the steady-state distribution of  Hawkes/GI/1 queues. Our algorithm is applicable to a variety of queueing models with Hawkes arrivals. In detail, we assume the arrival process is a linear Hawkes process, which covers both Markovian and non-Markovian Hawkes process as studied in \cite{Koops_JAP_2018} and \cite{Pender_wp_2018}, and the service times are i.i.d. following a general continuous distribution.

Our algorithm is closely related to the literature on perfect sampling of queueing models, for instance \cite{BlanchetChen_2015}, \cite{BlanchetChen_MOR_2019}, \cite{Dong}, \cite{Bouilard2014}, \cite{EnsorGlynn2000},  and \cite{Xiong2015}, to name but a few. Most of the existing works have been focused on queues with arrivals modeled as Poisson or renewal processes. In those cases, the steady-state waiting time can be related to the maximum or the running maximum of a (possibly multi-dimensional) random walk, see the discussion on page 378 of \cite{BlanchetChen_MOR_2019}. The running maximum of random walk is defined as the maximum from a positive time $n$ to infinity, for any $n$.  In our case, the steady-state waiting time of Hawkes/GI/1 queue can still be represented as $\max_{n\geq 0}R(n)$ for certain stochastic process $R(n)$. However, $R(n)$ is defined by a stationary version of the Hawkes arrivals and its increments have sequential dependence. Therefore, existing perfect sampling algorithms for queueing models can not be directly applied to queues with Hawkes arrivals. 

There are two key steps in our algorithm. First, to simulate the process $R(n)$ involves generating a stationary sample path of the Hawkes arrivals, namely, perfect sampling of Hawkes process,  and therefore is far from trivial. %The first problem we solve in our algorithm design, is how to do perfect sampling for Hawkes processes \textit{efficiently}. 
The only existing perfect sampling algorithm for Hawkes processes in the literature is developed in \cite{Rasmussen}. However, the complexity of the algorithm, in terms of the expected total number of random seeds generated, is infinite (see Propostion \ref{prop: rasmussen}). Using importance sampling techniques, we proposed a new, and probably the first perfect sampling algorithm for Hawkes processes that has finite expected termination time. In particular, the  complexity of our new algorithm is finite and has an explicit expression in terms of model and algorithm parameters. We believe this new algorithm can be applied to, in addition to the single-server queue studied in the current paper, other stochastic models that involve Hawkes processes as we mentioned previously. 

Once the stationary Hawkes process and $R(n)$ are simulated, the next key step of our algorithm is to find out $\max_{n\geq 0}R(n)$. The main idea is to construct a random walk $\tilde{R}(m)$ coupled with the Hawkes/GI/1 queue such that its running maximum $\max_{m\geq n}\tilde{R}(m)$ dominates the process $R(n)$ in a proper way. Then, we apply the techniques dealing with the running maximum of random walks, as developed in \cite{BlanchetChen_MOR_2019}, to simulate $\max_{n\geq 0}R(n)$ jointly with $\max_{m\geq n}\tilde{R}(m)$, which completes our algorithm as $\max_{n\geq 0}R(n)$ equals in distribution to the steady-state waiting time.

The rest in the paper is organized as follows. We first introduce the definition of Hawkes process, the Hawkes/GI/1 queue model and the technical assumptions in Section \ref{sec: model}. In Section \ref{sec: algorithm}, we introduce our perfect sampling algorithms for Hawkes processes and Hawkes/GI/1 queues, along with the main results of the paper. In Section \ref{sec: numerical}, we implement the algorithms and report the numerical experiment results. Finally, Section \ref{sec: conclusion} concludes the paper with a brief discussion on future research directions. Proofs of some technical results are included in the Appendices.

\section{Model and Assumptions}\label{sec: model}
\subsection{Hawkes Process}\label{subsec: Hawkes}
Following \cite{Koops_JAP_2018}, we provide two equivalent definitions for Hawkes process, namely the \textit{conditional intensity} and \textit{cluster representation} definitions. The conditional intensity definition clearly demonstrates self-excitement of Hawkes arrivals while the cluster representation provides an alternative probabilistic construction of Hawkes process, which will be used in our simulation algorithm.

\begin{definition}\label{def: Hawkes intensity} (Conditional Intensity) A Hawkes process is a counting process $N(t)$ that satisfies 
	\begin{equation*}
	P(N(t+\Delta t)-N(t)=m|\mathcal{F}(t))=
	\begin{cases}
	\lambda(t)\Delta t + o(\Delta t), & m=1\\
	o(\Delta t),& m>1\\
	1-\lambda(t)\Delta t + o(\Delta t)& m=0,
	\end{cases}
	\end{equation*}
	as $\Delta t\to 0$, where $\mathcal{F}(t)$ is the associated filtration and $\lambda$ is called the conditional intensity such that
	\begin{equation}\label{eq: Hawkes intensity}
	\lambda(t) = \lambda_0 + \sum_{i=-\infty}^{N(t)}h(t-t_i),
	\end{equation} 
	where $t_1, t_2$,... are the arrival times, the constant $\lambda_0 >0$ is called the \textit{background intensity} and the function $h:\mathbb{R}_+\to\mathbb{R}_+$ is called \textit{excitation function}.
\end{definition}

According to \eqref{eq: Hawkes intensity}, an arrival will increase the future intensity function and the possibility of new arrivals, so arrivals of a Hawkes process are self-exciting. 
Now we introduce another equivalent definition (or construction) of Hawkes process which represents it as a branching process with immigration.
\begin{definition}\label{def: Hawkes cluster}(Cluster Representation)  Consider a (possibly infinite) $T\geq 0$ and define a sequence of events $\{t_n\leq T\}$ according to the following procedure:
	\begin{enumerate}
		\item A set of immigrant events $\{\tau_m\leq T\}$  arrive  according to a Poisson process with rate $\lambda_0$ on $[0,T]$.
		\item For each immigrant event $\tau_m$, define a cluster $C_m$, which is a set of events, as follows. Each event $\in C_m$ is indexed by $k\geq 1$ and is represented by a tuple $(k, t_m^k,pa_m^k)$, where $t_m^k$ is the event's arrival time and $pa_m^k\geq 0$ is the index of its parent event. Following this representation, we denote the immigrant event as $(1,\tau_m,0)$. 
		\item The cluster $C_m$ is generated following a branching process. Initialize $k=1$, $C_m=\{(1,\tau_m,0)\}$.
		For event $(k,t_m^k,pa_m^k)\in C_m$, let $n=|C_m|$, namely cardinality of set $C_m$, generate a sequence of next-generation events $(n+1,t_m^{n+1},k),...,(n+\Lambda, t_m^{n+\Lambda},k)$ where $t_m^{n+1},...,t_m^{n+\Lambda}$ follow a non-homogeneous Poisson process on $[t_{m}^k,T]$ with rate function $\lambda(t)=h(t-t_{m}^k)$. Update $k\leftarrow k+1$ and add the newly generated events into set $C_m$. Repeat the above iteration until no more events are generated.
		\item Collect the sequence of events $\{t_n\} = \cup_m \{t_m^k: k=1,..., |C_m|\}$.
	\end{enumerate}
	Then, the counting process $N(t)$ corresponding to the event sequence $\{t_n\}$ is equivalent to the Hawkes process defined by conditional intensity function \eqref{eq: Hawkes intensity}.
	
\end{definition} 
%\textcolor{red}{Modify the representation of Hawkes process into the form of set, to make the description of algorithms more convenient.}

For each cluster $C_m$, we define a non-decreasing function $S_m(t)=|\{t_{m}^n: \tau_m\leq t_m^n\leq t+\tau_m\}|$, i.e. the number of events in $C_m$ that arrive on interval $[\tau_m,\tau_m+t]$ for $t\in\mathbb{R}$. By definition, for each $m$, $S_m(t)$ is an increasing function such that $S_m(t)\equiv 0$ for $t<0$, $S_m(0)=1$ and $S_m(\infty) = |C_m|$. Then, the Hawkes counting process $N(t)=\sum_{m=1}^\infty S_m(t-\tau_m)$. We define for each cluster $C_m$ its \textit{cluster length} $L_m$ as length of the time interval between the immigrant event and the last event in this cluster, i.e.
\begin{equation}\label{eq: L}
L_m \triangleq \max_{k}t_m^k - \tau_m.
\end{equation}  
We call $\tau_m$ and $\delta_m\triangleq \max_k t_m^k$  the \textit{arrival time}  and  \textit{departure time} of the cluster $C_m$, respectively.

Let $h_1 = \int_0^\infty h(t)dt$. Then, according to Definition \ref{def: Hawkes cluster}, $h_1$ is the expected number of next-generation events generated by each single event. For any event $t_m^k$ that is not an immigrant,  the arrival time of its parent event is $t_m^{pa_m^k}$ following our notation. We define the  \textit{birth time} $b_m^k$ of event $t_m^k$ as 
\begin{equation}\label{eq: b}
b_m^k \triangleq t_m^k-t_m^{pa_m^k}.
\end{equation}
Following the property of non-homogeneous Poisson process, conditional on $|C_m|$, $\{b_m^k: m\geq 1, 1\leq k\leq |C_m|\}$ are i.i.d. positive random variables that follow the probability density function $f(t)\triangleq h(t)/h_1$ for $t\geq 0$. Given the clear meaning of $h_1$ and $f(\cdot)$ in the cluster representation of Hawkes process, in the rest of the paper, we shall denote by $(\lambda_0, h_1, f(\cdot))$ as the parameters that decide the distribution of a Hawkes process.    \\

\textbf{Stationary Hawkes Process } For the Hawkes process with parameters $(\lambda_0, h_1,f(\cdot))$  to be stable in long term, intuitively, each cluster should contain a finite number of events on average. Therefore, we shall impose the following stability condition on the Hawkes process throughout the paper, which is also a common assumptions in the literature (\cite{Bremaud2002}, \cite{Hawkes1974}):
$$h_1<1.$$
Under this condition, the Hawkes process has a unique stationary distribution (\cite{Bremaud2002}). Actually, we can directly construct a stationary Hawkes process using the cluster representation as follows. Note that the arrival process of the immigrant, or equivalently, the clusters, is a homogeneous Poisson process and can be extended to time interval $(-\infty, \infty)$. For this two-ended Poisson process, we index the sequence of immigrant arrival times by $\{\pm 1, \pm 2,...\}$ such that $\tau_{-1}\leq 0<\tau_{1}$ and generate the clusters $\{C_{\pm m : m=1,2,...}\}$ independently for each $m$ following the procedure in Definition \ref{def: Hawkes cluster}. Then, the events that arrive after time $0$ form a stationary sample path of a Hawkes process on $[0,\infty)$, namely
$$N(t) \triangleq |\cup _{m=-\infty}^{\infty }\{(k,t_m^k,pa_m^k):k=1,2,...,|C_m|, 0\leq t_m^k\leq t\}| =\sum_{m=-\infty}^\infty (S_m(t-\tau_m)-S_m(-\tau_m)).$$
is a stationary Hawkes process.
% In the rest of the paper, for any two two time points $-\infty\leq a<b< \infty$, we denote
% $$C_m(a,b) = \{(k,t_m^k,pa(t_m^k),V_m^k)\in C_m, a<t_m^k\leq b\}, H(a,b)=\cup_{m=-\infty}^\infty C_m(a,b).$$

\subsection{$Hawkes/GI/1$ Queue}\label{subsec: HG1}
We consider a single-server queue where customers arrive according to a stationary Hawkes process with parameters $(\lambda_0,h_1,f(\cdot))$. We denote by $\{U_n\}$  the corresponding sequence of inter-arrival times. Upon arrival, customers are served FIFO (first-in-first-out) and their service times $V_n$ are i.i.d. positive random variables of general distribution with probability density function $g(\cdot)$. 

Denote by $W(t)$ the \textit{virtual waiting time} process (also known as the workload process) of this single-server queue. Mathematically, $W(t)$ can be defined as a reflected process as follows. Let $R(t) \triangleq \sum_{k=1}^{N(t)} V_k -t$ for all $t\geq 0$, then we have
$$dW(t) =dR(t) + dL(t), \text{ where } L(0)=0, dL(t)\geq 0\text{ and } W(t)dL(t)=0.$$
In the setting of a single-server queue, the function $L(t)$ is nothing but the server's idle time by time $t$. A single-server queue is said to be stable if the distribution of $W(t)$ converges as $t\to\infty$. The following proposition states that the $Hawkes/GI/1$ queue is stable if and only if its service rate is higher than the stationary arrival rate of customers.
\begin{proposition}\label{prop: stable}
	The $Hawkes/GI/1$ queue is stable if and only if
	\begin{equation}\label{eq: stable} 
	\lambda_0\cdot \frac{E[V_1]}{1-h_1} < 1.
	\end{equation}
\end{proposition}
\begin{proof}[Proof of Proposition \ref{prop: stable}]
	The stability condition directly follows Propositions 1.1 and 1.2 on page 267 of \cite{AsmussenBook}. For a single-server queue with stationary inter-arrival time sequence $\{U_n\}$ and  service time sequence $\{V_n\}$, it is stable if and only if $E[U_1]-E[V_1]>0$. In our case, $E[U_1] = (1-h_1)/\lambda_0$ and therefore, $E[U_1]-E[V_1]>0$ is equivalent to $\lambda_0\cdot \frac{E[V_1]}{1-h_1} < 1$.
\end{proof}
\vskip 2ex
For the Hawkes/GI/1 queue, each arrival of the Hawkes process is associated with a service time. Therefore, we shall also include service-time information in the cluster representation of Hawkes process. In detail, we denote $C_m=\{(k,t_m^k,pa_m^k,V_m^k): k=1,2,...,|C_m|\}$ where $V_m^k$ is the service time of the $k$-th customer in cluster $C_m$.

Our goal is to simulate the steady-state virtual waiting time for the Hawkes/GI/1 queue, so we shall impose the stability condition throughout the paper. Besides, in order to carry out importance sampling procedures in the simulation algorithm, we also need some technical assumptions on  the Hawkes/GI/1 model. These assumptions are satisfied by a large class of queueing models, such as Hawkes/GI/1 queue with arrivals that are Markovian Hawkes, or non-Markovian Hawkes having excitation function with finite support, and service times that are exponential or phase-time. Now we close this section by summarizing our assumptions on the Hawkes/GI/1 queue.
\begin{assumption}\label{asp: stable}
	The stability condition \eqref{eq: stable} holds.
\end{assumption}
\begin{assumption}\label{asp: light tail}
	The birth time of the Hawkes process $b$ and the service time of customers $V$ are of continuous distribution. Besides, there exists $
	\theta_0>0$ such that 
	\begin{equation*}
	E[\exp(\theta_0 b)]=\int_0^\infty \exp(\theta_0t)f(t)dt<\infty, ~E[\exp(\theta_0 V)] = \int_0^\infty\exp(\theta_0t)g(t)dt<\infty,
	\end{equation*}
	i.e. the random variables $b$ and $V$ have finite moment generating function in a neighborhood of the origin.
\end{assumption}	
\begin{assumption}\label{asp: tilting}
	The distributions of birth time and service time can be simulated exactly, and moreover, we can simulate from
	exponential tiltings (i.e., the natural exponential family) associated with these distributions.
\end{assumption}

\section{Simulation Algorithms}\label{sec: algorithm}
Our goal is to simulate the steady-state virtual waiting time $W_\infty$ of $Hawkes/GI/1$ queue. We now construct an expression of $W_\infty$ in the same spirit of \cite{Loynes}. First, we extend the \textit{stationary} Hawkes arrival process $N(t)$ \textit{backward in time} to $(-\infty]$, and  for $t<0$, define $N(t)$ as the number of arrivals on $[t,0]$. The i.i.d. service time sequence $\{V_n\}$ can be natural extended to $n\leq -1$. Following the same argument as in Proposition 1 of \cite{BlanchetChen_2015}, we can construct the stationary distribution $W_\infty$ as
\begin{equation}\label{eq: W infty}
W_{\infty}\stackrel{d}{=} \max_{t\geq 0} R(t), \text{ with } R(t)\triangleq\sum_{m=-1}^{-N(-t)} V_m -t.
\end{equation}
According to \eqref{eq: W infty}, we can simulate $W_\infty$ in two steps
\begin{enumerate}
	\item Generate the sample path of $R(t)$.
	\item Simulate the maximum $ \max_{t\geq 0}R(t)$. 
\end{enumerate}
To simulate the process $R(t)$ in Step 1 is essentially to simulate a \textit{stationary} sample path of the Hawkes process \textit{backward in time}.  We shall explain how to do this in Section \ref{subsec: algorithm Hawkes}, in which we develop a novel and efficient perfect sampling algorithm for Hawkes process. Then we explain Step 2 in Section \ref{subsec: algorithm Q} and this completes our algorithm.

\subsection{Perfect Sampling of Hawkes Process}\label{subsec: algorithm Hawkes}
%Now we explain how to simulate $H(-\infty,0)$. Since $H(-\infty,0)=\cup_{r=1}^\infty H(-r,-r+1)$, we essentially need to (1) simulate $H(-1,0)$; (2) for $r\geq 1$, simulate $H(-r-1,-r)$ conditional on $H(-r,0)$ sequentially.
%
%We now split all the clusters into a sequence of disjoint groups:
%\begin{equation*}
%\mathcal{C}_0=\{C_m:  -1\leq \delta_m\geq 0, \tau_m\leq 0\}, \mathcal{C}_r=\{C_m:  -r\leq \delta_m <-r+1\}, \text{ for }r\geq 1.
%\end{equation*}
%Then, 
%\begin{equation}\label{eq: H decompositon}
%\begin{aligned}
%H(-1,0)  &= \{(t,V)\in C_m(0,1)\}\\
%&= \{(t,V)\in C_m(-1,0): \delta_m <0\} \cup \{(t,V)\in C_m(-1,0):\delta_m\geq 0\}\\
%& =  \{(t,V)\in C_m(-1,0): -1\leq \delta_m <0\} \cup \{(t,V)\in C_m(-1,0):\delta_m\geq 0, \tau_m\leq 0\}\\
%&\triangleq H_1(-1,0) \cup H_0(-1,0),
%\end{aligned}
%\end{equation}
%where the third equality holds as clusters arrive after time 0 or depart before -1 have no impact on the Hawkes process on $(-1,0]$. To simulate  $H_1(-1,0)$ and $H_0(-1,0)$, it suffices to simulate two disjoint sets of clusters 

To the best of our knowledge, the only perfect sampling algorithm for Hawkes process  in the literature is given by \cite{Rasmussen}, which is based on the cluster representation of Hawkes process. However, this algorithm is not very efficient, and it on average needs to generate an infinite number of random numbers before termination (as we shall prove in Proposition \ref{prop: rasmussen}). In this part, we propose a novel perfect sampling algorithm for Hawkes process, whose complexity (in terms of the expected variables generated) is finite and has an explicit expression in model and algorithm parameters. Our algorithm is also based on the cluster representation, but exploits importance sampling and acceptance-rejection techniques to largely improve the simulation efficiency. 

\subsubsection{Existing Framework} As both our algorithm and \cite{Rasmussen} are based on cluster representation of Hawkes processes, we first briefly review the algorithm in \cite{Rasmussen} to introduce the outline of our algorithm. Then, we point out the major bottlenecks in their algorithm, and in Section \ref{subsubsec: Hawkes}, we provide our solutions to these bottlenecks and introduce the new algorithm.   

Recall that the arrival process of the immigrant events is a homogeneous Poisson and can be extended into a double-ended counting process $\{\tau_m\}_{m=-\infty}^\infty$. Let $M(t) =\max \{m:\tau_m\leq t\}$ be the index of the last immigrant that arrives by time $t$. The counting process $N(\cdot)$ corresponding to a stationary Hawkes process on $[0,\infty)$ can be decomposed as, for any $t>0$,
\begin{equation}
\begin{aligned}\label{eq: N0}
N(t) &=\sum_{m=-\infty}^\infty (S_m(t-\tau_m)-S_m(-\tau_m))\\
&= \sum_{m=-\infty}^{-1} \left(S_m(t-\tau_m) - S_m(-\tau_m)\right)+\sum_{m=1}^{M(t)} S_{m}(t-\tau_m)\\
&= \sum_{m\leq -1, L_m>-\tau_m}  \left(S_m(t-\tau_m) - S_m(-\tau_m)\right) +\sum_{m=1}^{M(t)} S_{m}(t-\tau_m)\\
&\triangleq ~~~~~~~~~~~~~~~~~~~N_0(t)~~~~~~~~~~~~~~~~~~~~~~~+~~~~~~N_1(t).
\end{aligned}
\end{equation}
Here $L_m$ is the cluster length as defined in \eqref{eq: L}. The second equality just says that if a cluster $C_m$ that arrives before time 0 has cluster length $L_m<-\tau_m$ (minus its arrival time), all of its events will arrived before 0, and therefore, it will have no impact on dynamic of the stationary Hawkes process after time 0. In the above decomposition, $N_1(\cdot)$ is the set of events from clusters arriving after time 0, which is basically a Poisson compound of i.i.d. clusters. To simulate $N_1(\cdot)$, we just need to simulate the arrivals of clusters according to a Poisson process of rate $\lambda_0$, and then, for each arrival, simulate a cluster independently according to Step 3 as described in Definition \ref{def: Hawkes cluster}. Therefore, the key step in perfect sampling of Hawkes process is to simulate $N_0(\cdot)$, or equivalently, to simulate all the clusters that have arrived before time 0 and last after time 0. By slightly notation abusing, we shall also use $N_0$ to denote the set of clusters that have arrived before time 0 and last after time 0, i.e., $N_0=\{C_m: m\leq -1, L_m>-\tau_n\}$

Let $p(t)=P(L_m>t)$ be the probability that a cluster last for more than $t$ units of time. Then, by Poisson thinning theorem, the clusters in $N_0$ arrive on time horizon $(-\infty,0]$ according a non-homogeneous Poisson process with intensity function $\gamma(t) = \lambda_0p(-t)$, for $t\leq 0$. Based on this observation, \cite{Rasmussen} proposes the following procedure to simulate $N_0$:\\

\noindent \textbf{Step 1:} Sample a non-homogeneous Poisson process with intensity function $\gamma(t)$ on $(-\infty,0]$ and obtain the arrivals $\{\tau_{-1},...,\tau_{-K}\}$. \\%($K$ is a Poisson r.v. with mean $\int_0^{\infty}\lambda_0p(-t)dt=\lambda_0E[L_m]$.)
\textbf{Step 2:} For each $k\in \{1,2,...,K\}$, repeatedly simulate a cluster sample $C$ until its cluster length $L>-\tau_k$. Set $C_{-k}=C$.\\
\textbf{Step 3:} Output $N_0=\{C_{-1}, C_{-2},...,C_{-K}\}$.\\

%The output sequence $\{(\tau_{-1}, C_{-1}), (\tau_{-2},C_{-2}),...,(\tau_{-K},C_{-K})\}$ then can be transferred to the values of $N_0(\cdot)$ according to \eqref{eq: N0}. 

The above algorithm has two major bottlenecks in the design that significantly affect its computational efficiency. First, in Step 1, the function $p(\cdot)$ does not have an explicit expression, so \cite{Rasmussen} uses a while loop to approximate $p(\cdot)$ by iteration. Second, in Step 2, a naive acceptance-rejection procedure is used to obtain a cluster with cluster length $>-\tau_k$. As a consequence, the total number of cluster simulation rounds, before an acceptance occurs, could be very large when $-\tau_k$ is large. The following proposition shows that in fact, Step 2 needs to spend, on average, infinite rounds of cluster simulation before termination.

\begin{proposition}\label{prop: rasmussen}
	Let $N_C$ be the total number of clusters generated in Step 2 by naive acceptance-rejection. Then,
	$$E[N_C]=\infty.$$
\end{proposition}
\begin{proof}[Proof of Proposition \ref{prop: rasmussen}]
	For any fixed  $t>0$, the expected number of clusters generated to obtain one sample with cluster length $L>t$ equals to $P(L>t)^{-1} = p(t)^{-1}$. Therefore, 
	$$E[N_C] =\int_0^{\infty}\gamma(-t)p(t)^{-1}dt = \int_0^{\infty}\lambda_0 p(t)p(t)^{-1}dt= \int_0^{-\infty}\lambda_0dt=\infty.$$
\end{proof}

\subsubsection{New Algorithm with Improved Efficiency}\label{subsubsec: Hawkes}
The bottleneck in Step 2 is indeed a rare event simulation problem, namely, for a given $t>0$, to simulate $C_m$ conditional on the event $\{L_m>t\}$ which could have very small probability for large $t$. In our new algorithm, we use importance sampling to design a more efficient procedure to sample from the conditional distribution of $C_m$ given $L_m>t$. Besides, we also combine importance sampling with a more sophisticated acceptance-rejection procedure to avoid evaluating $p(t)$ in Step 1, thus further improve the computational efficiency. 

We shall apply exponential tilting to do the importance sampling in Step 2. However, exponential tilting with respect to the cluster length $L_m$ is not easy. Instead, we shall do exponential tilting with respect to the \textit{total birth time} 
$$B_m\triangleq\sum_{k=2}^{|C_m|} b_m^k,$$
where $b_m^k$ is the birth time of event $k$ in cluster $C_m$ as defined in \eqref{eq: b}, and $b_m^1=0$ for the immigrant event. There are two reasons for us to use $B_m$ to do the exponential tilting. 
First, as the total birth time is always larger than cluster length,  for any cluster that arrives at time $-t$ to last after time 0, it must have total birth time $>t$. Therefore, to simulate $N_0$, it is sufficient to find the set of clusters $\{C_m: m\leq -1, B_m>-\tau_m\}$. Second, to apply importance sampling to a cluster $C_m$ using exponential tilting with respect to total birth time $B_m$ is much easier. In  Proposition \ref{prop: B} below, we summarize the properties of the random variable $B_m$ that are useful for our simulation algorithm, and explain explicitly how to do exponential tilting with respect to $B_m$. 

\begin{proposition}\label{prop: B}
	Consider a cluster  $C_m$ with parameter $(h_1,f(\cdot))$. Let $L_m$ and $B_m$ be its cluster length and total birth time, respectively. The following statements are true:
	\begin{enumerate}[(1)]
		\item $B_m\geq L_m$.
		\item Under Assumption \ref{asp: light tail}, there exists $\theta_1>0$ such that for any $0<\theta<\theta_1$, the cumulant generating function (c.g.f.) of $B_m$ is well-defined, i.e. $\psi_B(\theta)\triangleq\log E[\exp(\theta B_m)]<\infty$. Besides, for $0<\theta<\theta_1$, $\psi_B(\theta)$  satisfies:
		\begin{equation}\label{eq: B cgf}
		\psi_B(\theta)=h_1 \exp(\psi_f(\theta)+\psi_B(\theta))-h_1,
		\end{equation}
		with $\psi_f(\theta) = \log(\int_0^\infty e^{\theta t}f(t)dt)$ being the c.g.f. of the birth time.
		\item Let $\mathbb{P}$ be the probability distribution of a cluster $C_m$. Let $\mathbb{Q}$ be the importance distribution of the cluster under exponential tilting by parameter $0<\eta<\theta_1$ with respect to the total birth time $B_m$, i.e.
		$$d\mathbb{Q}(C_m) =\exp(\eta B_m-\psi_B(\eta))\cdot d\mathbb{P}(C_m).$$
		Then, sampling a cluster from the importance distribution $\mathbb{Q}$ is equivalent to sampling a cluster with parameter $(h_1\exp(\psi_B(\eta)+\psi_f(\eta)), f_\eta(\cdot))$ with $f_\eta(t) = f(t) \cdot\exp(\eta t - \psi_f(\eta))$.
	\end{enumerate} 
\end{proposition}

Given Proposition \ref{prop: B}, we are ready to give the whole procedure of our simulation algorithm for $N_0$ as described in Algorithm \ref{alg: hawkes}. The proof of Proposition \ref{prop: B} is given in Appendix \ref{apdx: proof B}.\\

\begin{algorithm}[h]
	\caption{Simulating $N_0$}
	\label{alg: hawkes}
	\begin{algorithmic}
		\REQUIRE parameters of the Hawkes Process $(\lambda_0, h_1,f(\cdot))$, a positive constant $0<\eta<\theta_1$
		\ENSURE  $N_0 =\{C_m: m\leq -1, L_m>-\tau_m\}$ \\
		\STATE 1. Compute $\psi_B(\eta)$ by solving Equation \eqref{eq: B cgf}.
		\STATE 2. Generate a non-homogeneous Poisson process on $(-\infty, 0]$ with rate function $$\tilde{\gamma}(t) = \lambda_0\exp(\psi_B(\eta)+\eta t),~ t\leq 0,$$
		~~~and obtain arrivals $\{\tau_1,...,\tau_{K}\}$.
		%In detail, first generate a Poisson random variable $\Gamma$ with mean $\int_0^{\infty}\gamma(-t)dt=\frac{\lambda_0}{\eta}\exp(\psi_B(\eta))$. For $k=1,2,...,K$, independently generate $\tau_k$ which is minus of an exponential r.v. with rate $\eta$. 
		\STATE 3. Initialize $N_0=\{\}$. 
		\STATE 4. For $m = 1,2,...,K$:\\
		~~~Generate a cluster $C_m$ with parameter $(h_1\exp(\psi_f(\eta)+\psi_B(\eta)),f_\eta(\cdot))$ and $X_m\sim U[0,1]$. \\ 
		~~~Accept $C_m$ and update $N_0 = N_0 \cup \{C_m\}$ if both of the following conditions are satisfied: 
		\begin{enumerate}[i.]
			\item $L_m>-\tau_m$,
			\item $X_m\leq \exp(-\eta(B_m+\tau_m))$.
		\end{enumerate}
		\STATE 5. Return $N_0$.
	\end{algorithmic}
\end{algorithm}

Algorithm 1 contains two importance-sampling steps. In Step 2,  instead of simulating the arrivals of clusters following the non-homogeneous Poisson with intensity function $\gamma(t)$, it simulates a non-homogeneous Poisson with larger intensity function $\tilde{\gamma}(t)\geq \gamma(t)$ (by Markov's Inequality). In Step 4, it applies importance sampling to generate the conditional distribution of clusters. In the end, it utilizes one step of acceptance-rejection to transform the two importance sampling probability laws jointly into the target distribution. 

The first main result of the paper is stated as Theorem \ref{thm1} below, in which we provide a theoretical guarantee that the output of Algorithm \ref{alg: hawkes} follows exactly the distribution of $N_0$, and an explicit expression of algorithm complexity in terms of model and algorithm parameters.

\begin{thm}\label{thm1}
	The list of clusters generated by Algorithm \ref{alg: hawkes} exactly follows the distribution of $N_0$. In particular, 
	\begin{enumerate}[(1)]
		\item The arrival times of the clusters follow a non-homogeneous Poisson process  with intensity $\gamma(t)=\lambda_0p(-t)$ for $t\in(-\infty,0]$. %The acceptance rate of a sample at time $-t$ equals to $\lambda(t)^{-1}P(K>t)$.
		\item For each cluster $C_m$ in the list, given its arrival time $\tau_m$, it follows the conditional distribution of a cluster given that the cluster length $>-\tau_m$.   %The accepted sample at time $-t$ follows the condition distribution of $K$ on $\{K>t\}$.
	\end{enumerate}
	Besides, the expected total number of random variables generated by Algorithm \ref{alg: hawkes} before termination is
	\begin{equation}\label{eq: complexity}
	\frac{\lambda_0\exp(\psi_B(\eta))(2-h_1-\psi_B(\eta))}{\eta(1-h_1-\psi_B(\eta))}.
	\end{equation} 
\end{thm}
\begin{proof}[Proof of Theorem \ref{thm1}]
	To prove Statement (1), by Poisson thinning theorem, it suffices to show that, for each $m$, the acceptance probability of the cluster $C_m$ equals to $$\frac{\gamma(\tau_m)}{\tilde{\gamma}(\tau_m)} = \mathbb{P}(L_m>-\tau_m)\exp(-\eta\tau_m)/\exp(\psi_B(\eta)).$$ According to Proposition \ref{prop: B}, the importance distribution $\mathbb{Q}$ and the target distribution $\mathbb{P}$ satisfies 
	$$d\mathbb{Q}(B_m = x) = d\mathbb{P}(B_m=x)\cdot\frac{\exp(\eta x)}{\exp(\psi_B(\eta))} .$$
	Therefore, the probability for cluster $C_m$ to be accepted in Step 3  is
	\begin{equation*}
	\begin{aligned}
	&E_\mathbb{Q}\left[1\left(L_m>-\tau_m\text{ and }X_m<\exp(-\eta (B_m+ \tau_m))\right)\right]\\
	=~& \int  1(L_m>-\tau_m) \exp(-\eta (B_m+ \tau_m))d\mathbb{Q}\\ 
	=~& \int  1(L_m>-\tau_m) \exp(-\eta (B_m+ \tau_m)) d\mathbb{P}\cdot\frac{\exp(\eta B_m)}{\exp(\psi_B(\eta))} \\
	= ~&\int  \frac{\exp(-\eta\tau_m) }{\exp(\psi_B(\eta))}1(L_m>-\tau_m) d\mathbb{P}\\
	=~&\mathbb{P}(L_m>-\tau_m)\exp(-\eta\tau_m)/\exp(\psi_B(\eta)).
	\end{aligned}
	\end{equation*}
	Therefore, we obtain Statement (1).
	
	From the above calculation, we can also see that, given $m$ and $\tau_m$, and any event $A\in \sigma(C_m)$, the joint probability
	$$P(C_m\in A, C_m\text{ is accepted}) = \int  \frac{\exp(-\eta\tau_m) }{\exp(\psi_B(\eta))}1(C_m\in A, L_m>-\tau_m) d\mathbb{P} \propto \mathbb{P}(L_m>-\tau_m, C_m\in A).$$
	Therefore, the accepted sample of $C_m$ indeed follows the conditional distribution of $C_m$ given $\{L_m> -\tau_m\}$, and we obtain Statement (2).
	
	To check \eqref{eq: complexity}, we first note that the expected total number of random variables generated by Algorithm \ref{alg: hawkes} is equal to the expected total number of clusters multiplied by the average number of random variables generated in one cluster. In Step 2, the number of clusters generated is a Poisson random variable with mean $\int_0^\infty \tilde{\gamma}(-t)dt = \lambda_0\exp(\psi_B(\eta))/\eta$. The average number of events in each cluster is 
	$$\frac{1}{1-h_1\exp(\psi_f(\eta)+\psi_B(\eta))} = \frac{1}{1-h_1-\psi_B(\eta)},$$
	where the last equality follows from \eqref{eq: B cgf}. Besides, for each cluster, the algorithm also need to simulate an extra random number in the acceptance-rejection step. Therefore, the expected total number of  random variables is 
	$$\frac{\lambda_0\exp(\psi_B(\eta))}{\eta}\cdot\left(1+\frac{1}{1-h_1-\psi_B(\eta)}\right)=\frac{\lambda_0\exp(\psi_B(\eta))(2-h_1-\psi_B(\eta))}{\eta(1-h_1-\psi_B(\eta))},$$ 
	which closes the proof.
\end{proof}	
\begin{remark}
	The complexity result \eqref{eq: complexity} not only guarantees that Algorithm \ref{alg: hawkes} terminates in finite time in expectation, it also provides some guidance to the optimal choice of $\eta$ that reduces the computational cost.
\end{remark}

\vskip 2ex

\textbf{Reversing the Time } So far we have focused on simulating a stationary Hawkes process forward in time. But to simulate the steady-state waiting time $W_\infty$ following \eqref{eq: W infty}, we need to generate the sample path of stationary Hawkes process \textit{backward in time} on  the interval $(-\infty, 0]$. We now explain briefly how to do this using Algorithm \ref{alg: hawkes}.

Recall that $\tau_m$ and $\delta_m$ are the arrival time and departure time of cluster $C_m$ as defined in Section \ref{subsec: Hawkes}. By definition, the cluster length $L_m=\delta_m-\tau_m$. Since $L_m$ are i.i.d. distributed and independent of $\tau_m$, by Poisson thinning theorem, $\{\delta_m\}$ also follows a homogeneous Poisson process of rate $\lambda_0$ just as $\{\tau_k\}$. Therefore, to simulate the Hawkes process backward in time, we can first apply Algorithm \ref{alg: hawkes} to simulate $N_0$, i.e. list of clusters that depart after time 0 but arrive before time 0, and then simulate those clusters that depart before time 0 according to a Poisson process with rate $\lambda_0$. In detail, we can use the following procedure to simulate a stationary Hawkes process backward in time on $[-t,0]$ for any $t>0$:
\vskip 1ex
%\textbf{Procedure to Simulate a Stationary Hawkes Backwards in Time: }
\begin{enumerate}
	\item Call Algorithm \ref{alg: hawkes} to simulate $N_0$.
	\item Simulate a Poisson process with rate $\lambda_0$ on $[-t,0]$ and obtain $0\geq \delta_1>...>\delta_K\geq -t$.
	\item For each $\delta_m$, simulate a cluster of events following Step 3 in Definition \ref{def: Hawkes cluster}. Adjust the event times accordingly such that the arrival time of the last event equals to $\delta_m$.
\end{enumerate}

Once a stationary Hawkes process can be simulated backward in time, we can simulate the process $R(t)$ in \eqref{eq: W infty} for any $t>0$. The next step is to find out $\max_{t\geq 0} R(t)$.

\subsection{Perfect Sampling of Single-Server Queue with Hawkes Arrivals}\label{subsec: algorithm Q}
In this part, we explain the second key step in our algorithm, namely, to simulate $\max_{t\geq 0}R(t)$. According to \eqref{eq: W infty}, once $\max_{t\geq 0}R(t)$ is simulated, we just return its value as an exact sample of $W_\infty$. 

Recall 
%that we construct the stationary distribution of $W_\infty$ as
%\begin{equation*}
%W_{\infty}\stackrel{d}{=} \max_{t\geq 0} R(t), \text{ with } R(t)\triangleq\sum_{m=-1}^{-N(-t)} V_m -t,
%\end{equation*}
%and 
that $\{U_n:n\leq -1\}$ is the sequence of inter-arrival times of the Hawkes process. Then,  for any $t\geq 0$,
$$-\sum_{m=-1}^{-N(-t)} U_n \geq -t,$$
because the left side is the arrival time of the first customer after time $-t$. Besides, the equality holds when $t$ is the arrival time of an event. Therefore,
$$\max_{t\geq 0} R(t) = \max_n \sum_{m=-1}^{-n}(V_m-U_m).$$
With slightly notation abusing, let's denote $R(n)\triangleq \sum_{m=-1}^{-n}(V_m-U_m)$.

For a  GI/GI/1 queue, the corresponding $\{R(n):n\geq 1 \}$ is basically a random walk with negative drift. For example, in \cite{EnsorGlynn2000}, the perfect sampling algorithm for GI/GI/1 is presented as a direct application of the simulation algorithm for the maximum of random walks with negative drift.  In our case, however, due to the self-exciting behavior of arrivals, $U_n$ has sequential dependence and  as a result, the distribution of $R(n)$ is more complicated than a random walk. We shall deal with this issue by constructing an auxiliary random walk $\tilde{R}(m)$ coupled with the Hawkes process, such that $\tilde{R}(m)$ has negative drift and its running maximum $\max_{m\geq n}\tilde{R}(n)$ ``dominates" the process $R(n)$ in a proper way. Then, we can bound and learn the exact value of $\max_{n\geq 0}R(n)$ by simulating the running maximum $\max_{m\geq n}\tilde{R}(m)$, applying the techniques developed in \cite{BlanchetChen_MOR_2019}.

\subsubsection{The Auxiliary Random Walk} We first explain our construction of the random walk $\tilde{R}(m)$. For a stationary Hawkes process backward in time, let's index its clusters in the order of their departure times. (In our previous notation, we index clusters in the order of their arrival times.) In detail, we shall denote the $m$-th cluster that depart before time $0$ as cluster $C_{-m}$ and its departure time as $\delta_{-m}<0$. Similarly, clusters that depart after time $0$ are also indexed by positive integers in the order of their departure times. We now denote by $A_{-k}\triangleq -\sum_{i=1}^k U_{-i}$ as the arrival time of the $k$-th customer before time 0. By definition, $\delta_{-m}$ is the arrival time of the last customer in cluster $C_{-m}$. As a result, for each $m\geq 1$, there must exist $k_m\geq 0$ such that
$$A_{-k_m}=\sum_{i=1}^{k_m}U_{-i} = -\delta_{-m}.$$
%It is easy to check that $k_m$ is increasing in $m$ and $k_m\to\infty$ as $m\to \infty$.
We define the auxiliary random walk $\tilde{R}(m)$ as the total service requests in cluster $C_{-1}$, ..., $C_{-m}$ minus the units of time that elapsed, i.e.
$$\tilde{R}(m)\triangleq \sum_{n=-1}^{-m}\sum_{k=1}^{|C_n|}V_n^k +\delta_{-m}.$$
Besides, for each $m$ we denote by $J(m)$ as the total service requests of customers who arrive before time $\delta_{-m}$ and belong to clusters that depart after time $\delta_{-m}$, i.e.
$$J(m) = \sum_{n\geq -m}\sum_{k=1}^{|C_n|}V_n^k \cdot1(t_n^k< \delta_{-m}).$$

The next proposition shows that the auxiliary random walk $\tilde{R}(m)$ has negative drift, and its increment ``dominates" the increment of $R(k)$ in a certain sense. Its proof is given in Appendix \ref{apdx: proof random walk}.
\begin{proposition}\label{prop: random walk}
	The following statements about the random walk $\tilde{R}(m)$ are true:
	\begin{enumerate}[(1)]
		\item $E[\tilde{R}(m)] <0$.
		\item For all $m_2\geq m_1\geq 1$, and any $k_{m_2}\leq k<k_{m_2+1}$, 
		$$R(k)\leq R(k_{m_1}) +J(m_1) +\tilde{R}(m_2) -\tilde{R}(m_1)$$
	\end{enumerate}
\end{proposition}

\subsubsection{Perfect Sampling Algorithm for $W_\infty$} The following corollary is a direct consequence of Proposition \ref{prop: random walk}, and provides a way to find $\max_{k\geq 0} R(k)$ by simulating the running maximum of the auxiliary random walk $\tilde{R}(m)$.
\begin{corollary}\label{crll1}
	Suppose there exist $m_2\geq m_1\geq 1$ such that the following statements are true:
	\begin{enumerate}[(a)]
		\item $\tilde{R}(m_2)-\tilde{R}(m_1)\leq - J(m_1)$;
		\item $\max_{m\geq m_2} \tilde{R}(m)-\tilde{R}(m_2)\leq \max_{0\leq k\leq k_{m_2}} R(k) - R(k_{m_1})$.
	\end{enumerate}
	Then, we can conclude
	$$\max_{k\geq 0} R(k) = \max_{0\leq k\leq k_{m_2}}R(k).$$
\end{corollary}
\begin{proof}[Proof of Corollary \ref{crll1}]
	For any $k\geq k_{m_2}$, suppose $k_m\leq k< k_{m+1}$ for some $m\geq m_2$. Following Statement (2) in Proposition \ref{prop: random walk}, we have
	\begin{equation*}
	\begin{aligned}
	R(k)&\leq R(k_{m_1})+J(m_1)+\tilde{R}(m)-\tilde{R}(m_1)\\
	&= R(k_{m_1}) + J(m_1)+(\tilde{R}(m)-\tilde{R}(m_2) )+(\tilde{R}(m_2)-\tilde{R}(m_1)) \\
	&\leq R(k_{m_1}) + J(m_1) + \max_{0\leq k\leq k_{m_2}} R(k) - R(k_{m_1}) -J(m_1)\\
	& = \max_{0\leq k\leq k_{m_2}} R(k).
	\end{aligned}
	\end{equation*}
	Therefore, $\max_{k\geq 0} R(k) = \max_{0\leq k\leq k_{m_2}}R(k)$.
\end{proof}

Corollary \ref{crll1} implies that we can stop simulating the stationary Hawkes process (backward in time) and return $W_\infty = \max_{0\leq k\leq k_{m_2}}R(k)$ when a pair of random times $m_1$ and $m_2$ satisfying Conditions (a) and (b)  are detected. For given $m_1$, and the Hawkes process up to cluster $C_{-m}$, it is straightforward to check whether $m$ satisfies Condition (a). Besides, since $\tilde{R}(m)$ is a random walk with strictly negative drift, the smallest $m_2$ that satisfying Condition (a) is finite. Condition (b) involves checking whether the running maximum of $\tilde{R}(m)$ exceeds a given level. %To check Condition (b), we will apply the technique developed by \cite{BlanchetChen_MOR_2019}. 
Based on the above observations,  we provide the outline of our perfect sampling algorithm for $W_\infty\stackrel{d}{=}\max_{k\geq 0}R(k)$ as describe in Algorithm \ref{alg: hawkes Q}

\begin{algorithm}[h]
	\caption{Perfect Sampling Algorithm for $W_\infty$}
	\label{alg: hawkes Q}
	\begin{algorithmic}
		\REQUIRE parameters of the Hawkes Process $(\lambda_0, h_1,f(\cdot))$,  service time distribution%a routine to simulate i.i.d. samples from the service time distribution and its natural exponential family.
		\ENSURE  $W_\infty$\\%, a random variable follows the stationary distribution of virtual waiting time.\\
		\STATE1. Call Algorithm \ref{alg: hawkes} to simulate the set of clusters $N_0$. 
		\STATE2. Set $m_1=0$, $m_2=0$, $M = 0$, compute $J=J(0)$ from $N_0$.
		\STATE3. Starting from $m_1$, keep simulating the random walk $\tilde{R}(m)$ for $m\geq m_1$ along with the corresponding clusters and their departure times until $m=m_2$ such that 
		$$\tilde{R}(m_2)-\tilde{R}(m_1)\leq -J.$$ 
		~~~~Compute $R(k)$ for $k_{m_1}\leq k\leq k_{m_2}$.\\
		~~~~Update %$n=n+ n_1$ and
		$M =\max_{0\leq k\leq k_{m_2}} R(k) - R(k_{m_1}) $.
		\item 4. Simulate a Bernoulli random variable $B$ with $P(B=1) = P(\max_{m\geq 0} \tilde{R}(m) > M)$. 
		\item If $B=0$,\\
		~~~~~\textbf{Return} $W_\infty = \max_{0\leq k\leq k_{m_2}} R(k)$.\\
		\item If $B=1$, \\
		~~~Simulate a sample path  $\tilde{R}^*(m)$ following the conditional distribution of $\tilde{R}(\cdot)$ given $\{\max_{m\geq 0} \tilde{R}(m) >M\}$ until $\Delta = \inf\{m: \tilde{R}^*(m)>M\}$. \\
		~~~~~~For $m_2+1\leq m\leq m_2+\Delta$, update $\tilde{R}(m) = \tilde{R}(m_2) +\tilde{R}^*(m-m_2)$. \\ 
		~~~~~~Update $m_1=m_2+\Delta$ and $J=J(m_1)$. \\
		~~~~~~Go back to Step 3.
	\end{algorithmic}
\end{algorithm}

In Step 4 of Algorithm \ref{alg: hawkes Q}, we directly apply the importance sampling technique in \cite{BlanchetChen_MOR_2019} to simulate the Bernoulli random variable jointly with the auxiliary random walk and to check Condition (b). To do this, we first need to verify that all assumptions in \cite{BlanchetChen_MOR_2019} are satisfied.
\begin{proposition}\label{prop: tilting}
	The following statements are true under Assumptions \ref{asp: stable} and \ref{asp: light tail}:
	\begin{enumerate}
		\item Define $K= \sum_{k=1}^{|C_{-1}|}V_{-1}^k$. Then, there exists $\theta_2>0$ such that for all $0\leq \theta<\theta_2$,
		$$\psi_K(\theta)\triangleq \log(E[\exp(\theta K)])<\infty.$$
		\item Recall that $C_m=\{(k,t_m^k,pa_m^k,V_m^k)\}$ as defined in Section \ref{subsec: HG1}. Let $\mathbb{P}$ be the probability distribution of $\{(C_m,\tau_m): m\leq -1\}$ generated by a Hawkes process with parameter $(\lambda_0, h_1, f(\cdot))$ and service time distribution $g(\cdot)$. Let $\mathbb{Q}$ be the importance  distribution of $\{(C_m,\tau_m):m\leq -1\}$  such that for all $n\geq 0$,
		$$d\mathbb{Q}((C_{-1},\tau_{-1}),...,(C_{-n},\tau_{-n})) \propto \exp(\eta\tilde{R}(n))\mathbb{P}((C_{-1},\tau_{-1}),...,(C_{-n},\tau_{-n})),$$
		for some constant $0\leq \eta<\theta_2$. Then, the distribution of $\{(C_m,\tau_m): m\leq -1\}$  under $\mathbb{Q}$,  can be generated by a Hawkes process with parameter set $\left(\lambda_0+\eta, h_1\exp(\psi_K(\eta)),f(\cdot) \right)$ and service time distribution $g_{\eta}(t)= g(t)\exp(\eta t -\psi_V(\eta))$.
	\end{enumerate}
\end{proposition}

The proof of Proposition \ref{prop: tilting} is given in Appendix \ref{apdx: proof tilting}. As a direct consequence of Proposition \ref{prop: tilting}, Assumptions A1) to A3) in \cite{BlanchetChen_MOR_2019} are satisfied and then, we can directly apply their algorithm to implement Step 4 of Algorithm \ref{alg: hawkes Q}. We refer the readers to \cite{BlanchetChen_MOR_2019} for more details and close this section by the second main result of this paper, which is a theoretical guarantee on the correctness and efficiency of Algorithm \ref{alg: hawkes Q}.
\begin{thm}\label{thm2}
	The output of Algorithm \ref{alg: hawkes Q} follows exactly the stationary distribution of virtual waiting time of the Hawkes/GI/1 queue. Besides, suppose $\bar{N}$ is the total number of random variables generated by Algorithm 2 before termination. Then, there exists $\delta>0$ such that $E[\exp(\delta N)]<\infty$.
\end{thm}
\begin{proof}[Proof of Theorem \ref{thm2}] The proof follows directly from proposition 3 and theorem 1 in \cite{BlanchetChen_2015}.\end{proof}
\section{Numerical Experiments}\label{sec: numerical}
We implement Algorithm \ref{alg: hawkes} and Algorithm \ref{alg: hawkes Q} in Python to test the performance and correctness of our perfect sampling algorithms. As an example of algorithm application, we investigate the effect of self-excitement on the steady-state waiting time distribution numerically using our perfect sampling algorithms.\\

\textbf{Algorithm \ref{alg: hawkes} Performance Test }We consider a Hawkes process with parameters $\lambda_0=1$, $h_1= 0.5$, $f(t) = 2\exp(-2t)$. Then, the stationary intensity rate of this Hawkes process is $\lambda_0/(1-h_1)=2$. To test the correctness of our algorithm, we apply Algorithm \ref{alg: hawkes} to simulate a stationary Hawkes process on time interval $[0,1]$. If our algorithm is correct, the average number of events generated on this time interval should equal to $E[N(1)]=2$, regardless of the choice of algorithm parameter $\eta$. To illustrate algorithm efficiency, we also record the number of random variables generated in each simulation round.  In Table \ref{tb: Hawkes}, we report, for different values of $\eta$, the 95\% confidence interval for $E[N(1)]$, along with the average number of random variables generated, based on 10000 rounds of simulation.
\begin{table}[h]
	\caption{95\% confidence interval for $E[N(1)]$ and expected number of random variables generated in one simulation round, estimated from 10000 i.i.d. sample path of Hawkes process on $[0,1]$ simulated by Algorithm \ref{alg: hawkes} with different $\eta$. }
	\label{tb: Hawkes}
	\centering
	\begin{tabular}{|c|c|c|}
		\hline
		$\eta$ & 95\%  Confidence Interval & \# random variables\\
		\hline
		0.05& $2.0079\pm 0.0397$&64.0009\\
		0.1& $2.0006\pm 0.0403$&34.7817\\
		0.15&$1.9713\pm 0.0395$&25.1570\\
		0.2&$1.9911\pm 0.0392$&21.6582\\
		0.35&$1.9950\pm 0.0406$&120.3092\\
		\hline
	\end{tabular}
\end{table}

\textbf{Algorithm \ref{alg: hawkes Q} Performance Test } We consider a Haweks/GI/1 queue in which customers arrive according to a Hawkes process with parameters $\lambda_0=1$, $h_1= 0.5$, $f(t) = 2\exp(-2t)$, and the service times are i.i.d. exponential with mean $1/3$. Since there is no theoretical benchmark to compare with, we shall compare the empirical distribution of the samples generated by Algorithm \ref{alg: hawkes Q} with that of the samples generated by simulating Hawkes/GI/1 for a long time such that the system is close to its steady sate. In figure \ref{fig: hist}, we plot the histograms of 10000 i.i.d. samples generated by Algorithm \ref{alg: hawkes Q} and 10000 i.i.d. samples of $W(100)$ obtained by simulating the single-server queue from empty state, i.e. $W(0)=0$. When simulating $W(100)$, to mitigate the transient bias caused by Hawkes arrivals, we feed in the single-sever queue with stationary sample paths of Hawkes process generated by Algorithm \ref{alg: hawkes}.

\begin{figure}[h]	
	\caption{Comparison of the estimated distributions of $W_\infty$ from 10000 i.i.d. samples generated by perfect sampling Algorithm \ref{alg: hawkes Q} and  long-term simulation respectively.}
	\label{fig: hist}
	\centering
	\includegraphics[scale = 0.7]{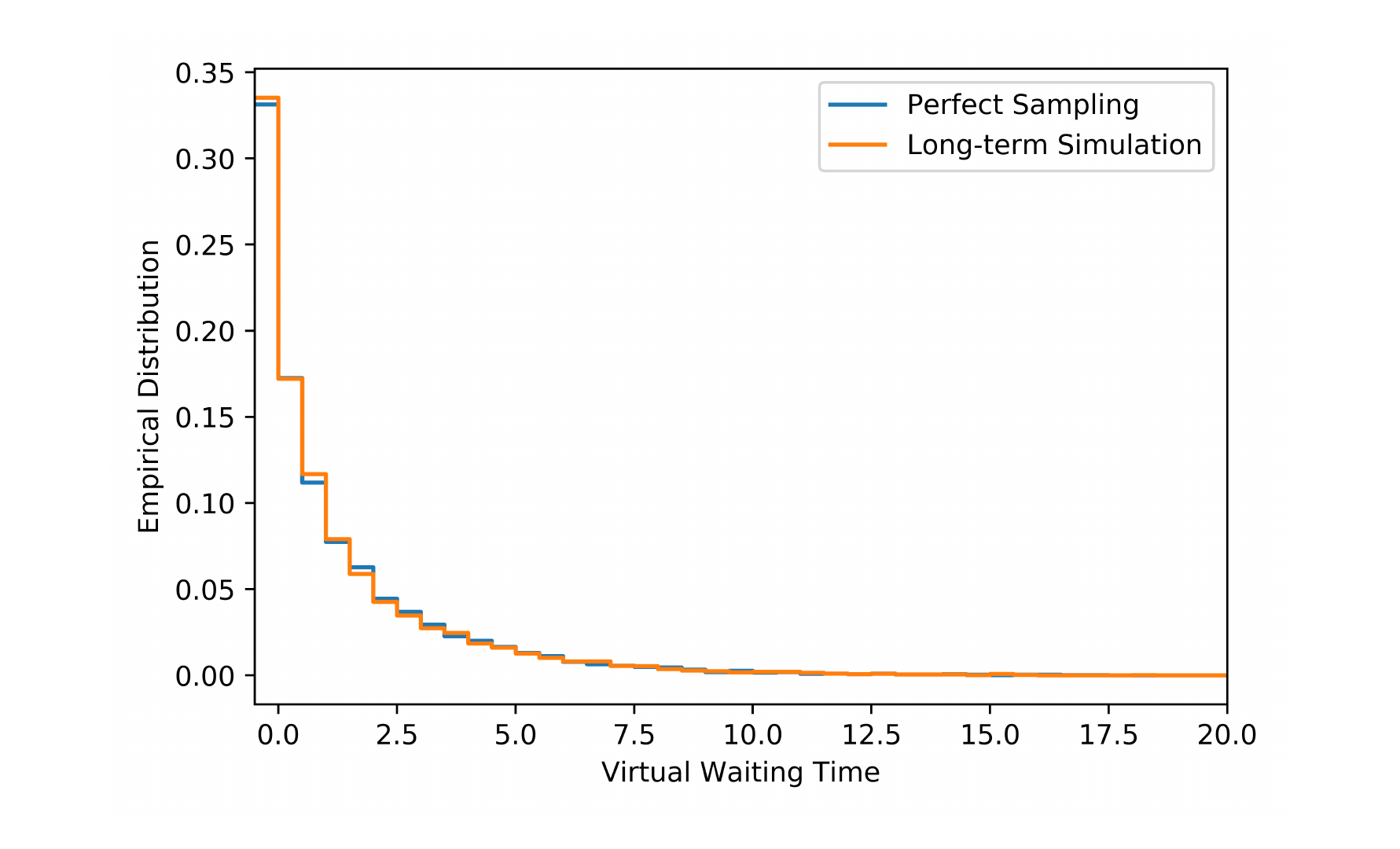}
\end{figure}

In order to compare the efficiency of our algorithm with naive simulation, in each simulation round of Algorithm \ref{alg: hawkes Q}, we also record $T_{ps}=-\tau_{-m_2+\Delta}$ which is minus of the departure time of the last cluster generated by the algorithm before termination. Basically, $T_{ps}$ is length of the sample path generated by Algorithm \ref{alg: hawkes Q} and we call $T_{ps}$ the PS sample path length.  In this light, $E[{T}_{ps}]$ can be used as a measurement for the computational cost of Algorithm \ref{alg: hawkes Q}. To evaluate the efficiency of naive simulation, we estimate the mixing time of $Hawkes/GI/1$ queue by comparing the estimated value of $E[W(T)]$ by naive simulation, for $T= 5, 10,..., 100,$ with that of $E[W_\infty]$ by Algorithm \ref{alg: hawkes Q}, using 10000 simulation rounds respectively. Figure \ref{fig: mixing} indicates that $E[W(T)]$ is not close to $E[W_\infty]$ for $T< 40$ while the sample average of PS sample path length, $\bar{T}_{ps} \approx 28.4936$ (the vertical dashed line). The numerical results indicate that, at least in this particular example,  the expected sample path length generated by Algorithm \ref{alg: hawkes Q} is smaller than the minimum sample path length with which a naive simulation could approximate the steady state expectation with good accuracy.

\begin{figure}[h]
	\caption{(a) Comparison between the mixing time of Hawkes/GI/1 queue and the average sample path length generated by Algorithm \ref{alg: hawkes Q}. (b) Empirical distribution of sample path length generated by Algorithm \ref{alg: hawkes Q}. }
	\centering
	\subfigure[mixing time v.s. PS sample path length ]{
		\includegraphics[scale=0.45]{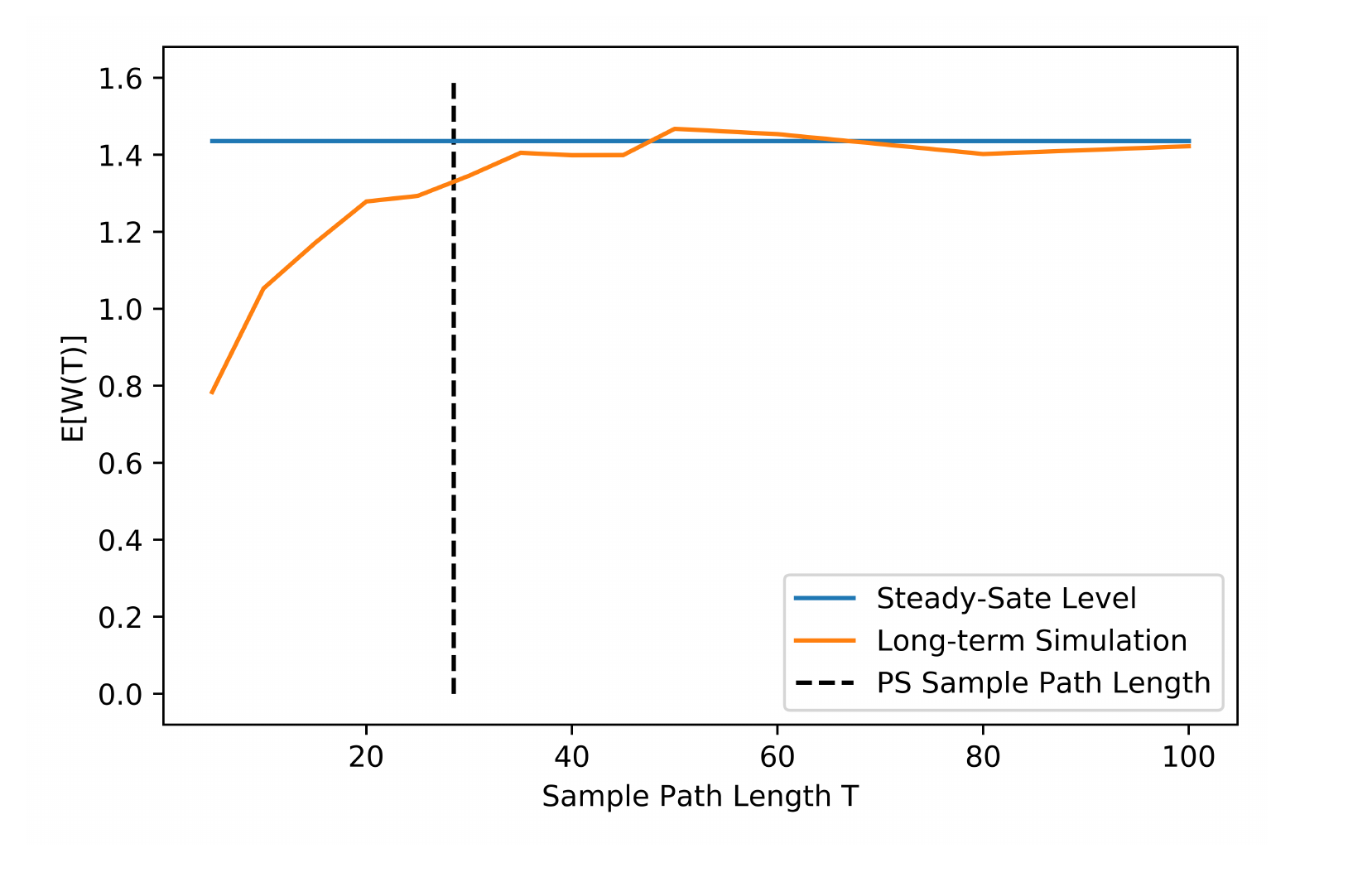}
		\label{fig: mixing}
	}
	\subfigure[PS sample path length distribution ]{
		\includegraphics[scale = 0.45]{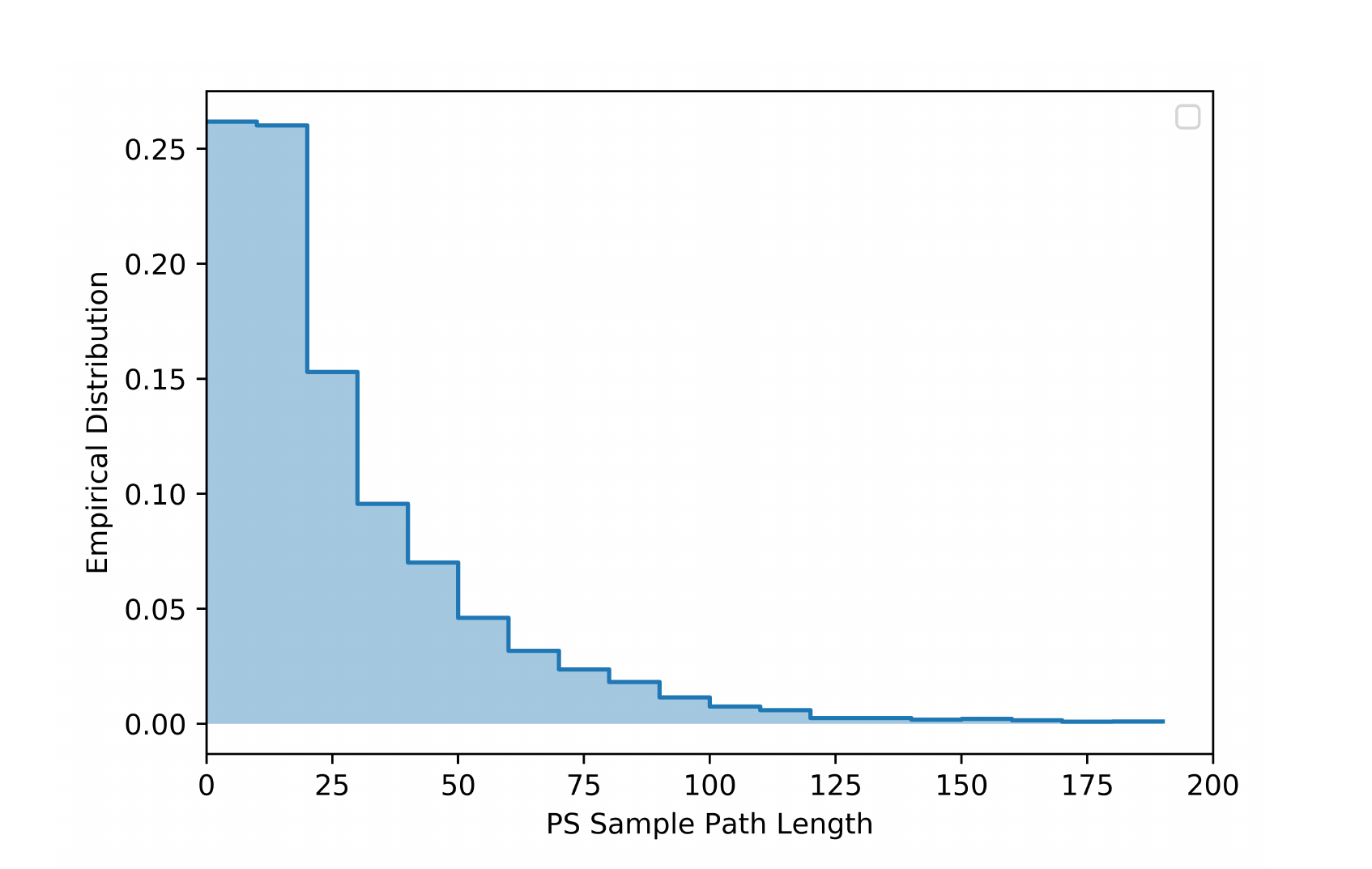}
		\label{fig: PT}
	}
\end{figure}

We conjecture that, this is because, the perfect sampling algorithm can terminate much earlier before the time point $T$ when the system get close to the steady state on average, say $T=40$, if the simulated sample path coalesces with the steady state, i.e. when a Bernoulli $B=0$ is detected in Step 4 of Algorithm \ref{alg: hawkes Q}, very soon. We plot the empirical distribution of $T_{ps}$ in Figure \ref{fig: PT} and find that more than 50\%  of the sample paths are shorter than 20. On the other hand, a significant proportion (over 10\%) are longer than 50. Intuitively, these sample paths could contribute to the average transient bias of $E[W(T)]$ for all $T< 50$ and thus affect the accuracy of naive simulation with fixed $T<50$.\\

\textbf{Impact of Self-excitement on Waiting Time }
Now we apply our simulation algorithms to estimate the impact of self-excitement on the steady-state distribution of waiting times. In particular, we investigate a set of Hawkes/GI/1 queues with equal stationary customer arrival rate and same service time distribution,  but different levels of self-excitement in customer arrivals. To see the impact of self-excitement on the steady-state distribution of the virtual waiting time,  we apply Algorithm \ref{alg: hawkes Q} to estimate $E[W_\infty]$ and $Var(W_\infty)$.

In detail, we consider 5 Hawkes/GI/1 queues indexed by $i=1,2,..,5$. The parameter set of the Hawkes process in queue $i$ is $(\lambda_0^i, h_1^i, f^i(\cdot))$. We set $f^i(t)=2\exp(-2t)$ for all $i$ and $h_1^i\in \{0.3,0.4,0.5,0.6,0.7\}$. For each $i$, $\lambda_0^i$ satisfies $\lambda_0^i/(1-h_0^i)\equiv 2$ so that all 5 Hawkes/GI/1 queues have equal stationary customer arrival rate. We also assume that the service times are exponential with rate $\mu=3$ in all queues. For each queue, we run 10000 rounds of simulation using Algorithm \ref{alg: hawkes Q} and report the estimated $E[W_\infty]$ and $Var(W_\infty)$ in Table \ref{tab: 5Q}. From the simulation results, we can see that self-excitement behavior of customer arrivals could increase not only the mean of waiting time but also its level of dispersion measured by the variance-to-mean ratio.
\begin{table}[h]
	\caption{Estimated mean, variance and VMR (variance-to-mean ratio) of steady-state virtual waiting time for Hawkes/GI/1 queues with different level of self-excitement.}
	\label{tab: 5Q}
	\centering
	\begin{tabular}{|c|ccccc|}
		\hline
		$h_1$& 0.7&0.6&0.5&0.4&0.3\\
		\hline
		$E[W_\infty]$&3.4388&2.0635&1.4356&1.1018&0.9287\\
		$Var[W_\infty]$&34.8779&11.2050&5.0345&2.6553&1.9101\\
		VMR&10.1425&5.4301&3.5069&2.4100&2.0567\\
		\hline
	\end{tabular}
\end{table}
\section{Conclusion}\label{sec: conclusion}
In this paper, we develop a simulation algorithm to generate i.i.d. samples exactly from the steady state of   $Hawkes/GI/1$ queues. To the best of our knowledge, this is the first perfect sampling algorithm for queueing models with arrivals that have self-exciting behavior. As a key component fo the algorithm, we also develop a new perfect sampling algorithm for Hawkes process that is much more efficient compared to existing algorithm in the literature. Both algorithms utilize importance sampling techniques on the cluster representation of Hawkes process. This approach, we believe, can probably be extended to other classes of counting processes with cluster representation, such as multi-dimensional Hawkes process with mutual-excitement and Poisson cluster processes. Besides, as single-server queues are the basic building blocks of more complicated queueing models, our approach can probably lead to perfect sampling methods for other queueing models with self-exciting arrivals.

%\bibliographystyle{plain}
%\bibliography{references}

\appendix
\section{Proof of Proposition \ref{prop: B}}\label{apdx: proof B}
We prove the three statements of Proposition \ref{prop: B} one by one.

(1). To see $B_m\geq L_m$, probably the most straightforward way is to represent the cluster as a tree. Let the immigrant event $\tau_m$ be the root node and link each event $t_m^k$ to its parent event by an edge of length $b_m^k$. Then, by their definitions, $B_m$ is equal to the total length of all edges in the tree while $L_m$ is equal to the length of the longest path(s) from the root node to a leaf node. Therefore, $B_m\geq L_m$.

(2). Recall that $B_m=\sum_{k=2}^{|C_m|} b_m^k$ and $b_{m}^k$ are i.i.d. for given $|C_m|$. Let $S_m=|C_m|$. It is known in literature (\cite{Bordenave2007}) that the c.g.f. $\psi_S(\theta)$ of $S_m$ is well-defined in a neighborhood around $0$. Under Assumption \ref{asp: light tail}, the c.g.f. of  each $b_m^k$ is well-defined on $[0,\theta_0]$. Since $B_m$ is the compound sum of $(S_m-1)$ i.i.d. birth times, $\psi_B(\theta)$ is also well-defined in a neighborhood around $0$. 

To obtain \eqref{eq: B cgf}, recall that $\tau_m=t_m^1$ is the immigrant event and suppose its next-generation events are  $t_m^2$, $t_m^3$,..., $t_m^{\Lambda+1}$. Then, $\Lambda$ is a Poisson r.v. with mean $h_1$. By the self-similarity of branching process, the total birth time $B_m$ equals to $\sum_{k=1}^{\Lambda} (t_m^{k+1}-t_m^1)$ plus i.i.d. copies of total birth times $B^{(k)}_m$ of the ``sub-clusters" brought by the next-generation events $t_m^{k+1}$ for $k=1,...,\Lambda$. Therefore, for any $0\leq \theta<\theta_1$, we have
\begin{equation*}\label{eq: prop B}
\begin{aligned}
\psi_B(\theta) &= \log(E[\exp(\theta B)]) = \log\left(E\left[\exp\left(\theta\sum_{k=1}^{\Lambda} \left(t_m^{k+1}-t_m^1 +B^{(k)}_m\right)\right)\right]\right)\\
&= \log\left(E\left[E\left[\exp\left(\theta\sum_{k=1}^{\Lambda} \left(b_m^{k+1} +B^{(k)}_m\right)\right)| \Lambda \right]\right]\right)\\
&= \log\left(E\left[\exp\left( \Lambda(\psi_f(\theta) + \psi_B(\theta) \right))\right]\right)\\
%& = \psi_{\Gamma}(\psi_f(\theta)+\psi_B(\theta))\\
&=h_1 \exp(\psi_f(\theta)+\psi_B(\theta))-h_1,
\end{aligned}
\end{equation*}
the last equality follows from the fact that $\Lambda$ is a Poisson r.v. with mean $h_1$.

(3). Under $\mathbb{Q}$, the c.g.f. of $B$ becomes $\psi_{B,\eta}(\theta)=\psi_B(\theta+\eta)-\psi_B(\eta)$. We can compute,
\begin{equation*}
\begin{aligned}
\psi_{B,\eta}(\theta)&=\psi_B(\theta+\eta)-\psi_B(\eta) =  h_1 \exp(\psi_f(\theta+\eta)+\psi_B(\theta+\eta))-h_1 \exp(\psi_f(\eta)+\psi_B(\eta))\\
&= h_1 \exp(\psi_f(\eta)+\psi_B(\eta)) (\exp(\psi_{f,\eta}(\theta)+\psi_{B,\eta}(\theta))-1),
\end{aligned}
\end{equation*}
with $\psi_{f,\eta}(\theta)=\psi_f(\eta+\theta)-\psi_f(\eta)$ be the c.g.f.  corresponding to probability density function $f_\eta(\cdot)$. The above calculation shows that $\psi_{B,\eta}(\theta)$ equals exactly to the c.g.f. of the total birth time  of a cluster with parameter $( h_1 \exp(\psi_f(\eta)+\psi_B(\eta)), f_\eta(\cdot))$. Therefore, the cluster under $\mathbb{Q}$ is equal in distribution to a cluster with parameter $( h_1 \exp(\psi_f(\eta)+\psi_B(\eta)), f_\eta(\cdot))$ (see Chapter 5.1 of \cite{SimulationBook}).

\section{Proof of Proposition \ref{prop: random walk}}\label{apdx: proof random walk}
By definition, the increment of $\tilde{R}(m)$ equals to the total service requirement in one cluster, minus the inter-departure time of clusters, i.e. an exponential random variable with rate $\lambda_0$. Therefore,
$$E[\tilde{R}(m)] = m\left(E[V]E[|C_m|]-\frac{1}{\lambda_0}\right)<0,$$
where the last inequality follows from Assumption \ref{asp: stable}. On the other hand,
\begin{equation*}
\begin{aligned}
R(k) -R(k_{m_1}) &=\sum_{j=k_{m_1}+1}^{k} V_{-j} - \sum_{j=k_{m_1}+1}^k U_{-j}\\
&=\sum_{n\geq -m_2} \sum_{k=1}^{|C_n|}V_n^k1(A_k\leq t_n^k< \delta_{-m_1})- \sum_{j=k_m+1}^k U_{-j}\\
&=\sum_{n\geq -m_1} \sum_{k=1}^{|C_n|}V_n^k1(A_k\leq t_n^k< \delta_{-m_1}) +\sum_{n=-m_1-1}^{-m_2}\sum_{k=1}^{|C_n|}V_n^k1(A_k\leq t_n^k)- \sum_{j=k_m+1}^k U_{-j}\\
&\leq \sum_{n\geq-m_1} \sum_{k=1}^{|C_n|}V_n^k1(t_n^k< \delta_{-m_1}) +\sum_{n=-m_1-1}^{-m_2}\sum_{k=1}^{|C_n|}V_n^k - \sum_{j=k_{m_1}+1}^{k_{m_2}} U_{-j}\\
&= J(m_1)  + \tilde{R}(m_2)-\tilde{R}(m_1 )
\end{aligned}
\end{equation*}

\section{Proof of Proposition \ref{prop: tilting}}\label{apdx: proof tilting}
The first statement follows a similar argument in the proof of Statement (2) of Proposition \ref{prop: B} as $\tilde{R}(1) =\sum_{k=1}^{|C_{-1}|}V_{-1}^k-\tau_{-1}$ is a compound sum of i.i.d. service times minus an independent exponential random variable. 

The second statement also follows a similar argument as used in the proof of Statement (3) of Proposition \ref{prop: B} by computing c.g.f. of the exponential tilting. In particular, define $\psi_R(\theta)=\log(E[\exp(\tilde{R}(1)\theta)])$ and $\psi_{R,\eta}(\theta) = \psi_R(\theta+\eta)-\psi_R(\eta)$. Recall that $K= \sum_{k=1}^{|C_{-1}|}V_{-1}^k$, then
$$\psi_R(\theta) = \log(E[\exp(\theta\tilde{R}(1))])= \psi_K(\theta) - \log\left(\frac{\lambda_0}{\lambda_0+\theta}\right).$$ 
Let $K^{(k)}$ be i.i.d. copies of $K$ for $k\geq 1$ and $\Lambda$ be the number of next-generation events generated by the immigrant. Then, by self-similarity of the branching process,
\begin{equation*}
\begin{aligned}
\psi_K(\theta)&= \log\left(E\left[\exp\left(\theta V_{-1}^1+\theta\sum_{k=1}^{\Lambda}  K^{(k)}\right)\right]\right) \\%= \log\left(E\left[E\left[\exp\left(\sum_{k=1}^{\Lambda}  V_{-1}^k +K^{(k)}\right)| \Lambda \right]\right]\right)\\
%&= \log\left(E\left[\exp\left( \Lambda(\psi_f(\theta) + \psi_B(\theta) \right))\right]\right)\\
%& = \psi_{\Gamma}(\psi_f(\theta)+\psi_B(\theta))\\
&=\psi_V(\theta) + h_1(\exp(\psi_k(\theta))-1),
\end{aligned}
\end{equation*}
where $\psi_V$ is the c.g.f. of the service time. Therefore,
\begin{equation*}
\begin{aligned}
\psi_{K,\eta}(\theta)&=\psi_K(\theta+\eta)-\psi_K(\eta) \\
& = \psi_V(\theta+\eta)-\psi_V(\eta) + h_1\exp(\psi_K(\eta))(\exp(\psi_{K}(\theta+\eta)-\psi_K(\eta))-1)\\
%&=  h_1 \exp(\psi_V(\theta+\eta)+\psi_K(\theta+\eta))-h_1 \exp(\psi_V(\eta)+\psi_K(\eta)) \\
%&= h_1 \exp(\psi_V(\eta)+\psi_K(\eta)) (\exp(\psi_{V,\eta}(\theta)+\psi_{K,\eta}(\theta))-1),
& = \psi_{V,\eta}(\theta) + h_1\exp(\psi_K(\eta))(\exp(\psi_{K,\eta}(\theta))-1)
\end{aligned}
\end{equation*}
and as a result,
$$\psi_{R,\eta}(\theta)= \psi_{V,\eta}(\theta) + h_1\exp(\psi_K(\eta))(\exp(\psi_{K,\eta}(\theta))-1)-\log\left(\frac{\lambda_0+\eta}{\lambda_0+\eta+\theta}\right).$$
In other words, the $\psi_{R,\eta}$ equals to the c.g.f. of $\tilde{R}(1)$ generated by Hawkes clusters with arrival rate $\lambda_0+\eta$, branching parameter $h_1\exp(\psi_K(\eta))$ and  service distribution $g_{\eta}(t) = g(t)\exp(\eta t-\psi_V(\eta))$.

\end{document}